\documentclass[twoside,leqno,12pt]{amsart} 
\usepackage{amsfonts} 
\usepackage{amsmath} 
\usepackage{amscd} 
\usepackage{amssymb} 
\usepackage{amsthm} 
\usepackage{url}
\usepackage{latexsym} 
\usepackage{color} 

\usepackage{xcolor}

\usepackage[linktoc=all,colorlinks,linkcolor=teal,citecolor=blue]{hyperref}

\newtheorem{theorem}{Theorem} 
\newtheorem{lemma}{Lemma}

\newtheorem{remark}{Remark}
\numberwithin{equation}{section}

\usepackage{bm, amssymb}
\usepackage{mathrsfs}

\begin{document}

\title{Average shifted convolution sum for $\mathrm{GL}(d_1)\times \mathrm{GL}(d_2)$}

\author[E. A. Molla] {Esrafil Ali Molla}
\address{Esrafil Ali Molla\\
	Indian Statistical Institute, Kolkata\\
	Stat-Math Unit\\
	203, B.T Road, Baranagar, West Bengal~700108\\
	India}
\email{esrafil.math@gmail.com}

\date{\today}



\begin{abstract}
We study the average shifted convolution sum  
$$
B(d_1,d_2;H,N):= \frac{1}{H} \sum_{h \sim H} \sum_{n \sim N} A_{\pi_1}(n)\, A_{\pi_2}(n+h),
$$
where $A_{\pi_i}(n)$ denotes the  Fourier coefficients of a Hecke--Maass cusp form $\pi_i$ for $\mathrm{SL}(d_i,\mathbb{Z})$ with $d_i\ge 4$, $i=1,2$. We establish a nontrivial power-saving bound of $B(d_1,d_2;H,N)$ for the range of the shift $H\ge N^{1-\frac{4}{d_1+d_2}+\varepsilon}$ for any $\varepsilon>0$. For the cases $d_1 = d_2 + 1$ and $d_1 = d_2$, our result improves a result that can be derived from a theorem of Friedlander and Iwaniec \cite{FI05}. In particular, when $d_1 = d_2$, we reach the critical threshold $H\ge N^{1-2/d+\varepsilon}$ such that any further improvement in this range yields a subconvexity bound for the corresponding standard $L$-function in the $t$-aspect.
\end{abstract}

\maketitle

\tableofcontents

\section{Introduction and main results}

The study of shifted convolution sums is of central importance in analytic number theory and has a long history. Since Selberg's seminal work \cite{Selberg65}, several authors have extensively investigated shifted convolution sums involving $\mathrm{GL}(2)$ Fourier coefficients. Any nontrivial bound of the shifted convolution sum has significant implications in the subconvexity problems and the equidistribution aspects of quantum unique ergodicity; for more details, see \cite{Blom04}, \cite{BlHa08}, \cite{DFI93}, \cite{DFI94}, \cite{Holo09}, \cite{Michel04} and \cite{Sarnak01}. Though several strong results have been established in the rank-two and rank-three settings, that is, for the $\mathrm{\mathrm{GL}}(2)$ and $\mathrm{\mathrm{GL}}(3)$ cases, the problem has remained wide open in higher rank. To the best of our knowledge, this article is the first to address the shifted convolution problem in the general $\mathrm{\mathrm{GL}}(d)$ setting for $d \ge 4$. Pitt \cite{Pitt95} has considered a similar sum with $\tau_3(n)$, the ternary divisor function, defined to be the Dirichlet coefficients of $\zeta^3(s)$ in the half-plane $\Re(s) > 1$, and $\lambda_f(n)$ denotes the  $n$-th Fourier coefficients of an $\mathrm{SL}(2,\mathbb{Z})$ Hecke cusp cusp form $f$. For any $0<r<N^{1/24}$ and any $\varepsilon>0$, he established that
$$
\sum\limits_{n\le N} \tau_3(n)\lambda_f(rn-1)\ll_{f,\varepsilon} N^{1-1/72+\varepsilon}.
$$
Recently, Munshi \cite{Munshi13} replaced $\tau_3(n)$ by the Fourier coefficients of an $\mathrm{SL}(3,\mathbb{Z})$ Hecke–Maass form. Utilizing a variant of Jutila's circle method with an important new input, factorizable moduli, allowed him to more effectively balance the diagonal and off-diagonal contributions, thereby obtaining the first non-trivial bound $N^{1 - 1/26 + \varepsilon}$ for the below sum. 
$$
\sum_{m=1}^{\infty} A_\pi(1,m)\,\lambda_f(m + h)\,V\left(\frac{m}{X}\right) \ll_{\pi,f,\varepsilon} X^{1 - \frac{1}{26} + \varepsilon}.
$$
Following Munshi's approach, Xi \cite{Xi18} achieved a strong power saving $N^{1 - 1/22 + \varepsilon}$ by exploring the bilinear structure of the exponential sum.

Baier, Browning, Marasingha, and Zhao \cite{BBMZ12} considered the average of the shifted convolution sum with $\lambda_f(n)$ replaced by $\tau_3(n)$, which is relevant for the sixth moment of the zeta function. Harun and Singh considered both the average and weighted average versions of shifted convolution sums for $\mathrm{GL}(3) \times \mathrm{GL}(2)$ in \cite{Harun24a}, and for $\mathrm{GL}(3) \times \mathrm{GL}(3)$ in \cite{Harun24b}. Recently, Dasgupta, Leung, and Young improved the lower bound $H \ge N^{1/4 +\varepsilon}$ for $\mathrm{GL}(3)$ Fourier coefficients  (see Theorem~1.2 of \cite{DLY24}). Subsequently, Pal and Pal \cite{PalPal} refined the result of Dasgupta \emph{et al.} \cite{DLY24}, improving the lower bound of $H$ to $N^{1/6+\varepsilon}.$ 

We consider an average of a shifted convolution sum for the general higher rank analogue, namely,
\begin{equation} \label{def of B(d_1,d_2;H,N)}
B(d_1,d_2;H,N):= \frac{1}{H} \sum_{h \sim H} \sum_{n \sim N} A_{\pi_1}(n)\, A_{\pi_2}(n+h),
\end{equation}
where $A_{\pi_i}(n)$ denotes the  Fourier coefficients of a Hecke--Maass cusp form $\pi_i$ for $\mathrm{SL}(d_i,\mathbb{Z})$ with $d_i\ge 4$, for $i=1,2$. Throughout the sequel, we assume without loss of generality that $d_1 \ge d_2$. By applying Cauchy’s inequality together with standard bounds for the $\mathrm{GL}(d_1)$ and $\mathrm{GL}(d_2)$ Fourier coefficients derived from Rankin–Selberg theory, we obtain the trivial estimate
$$
B(d_1,d_2;H,N)\ll N^{1+\varepsilon},
$$
where the implied constant depends only on $\pi_1,\, \pi_2$ and $\varepsilon$.

In this article, we prove a non-trivial bound for $B(d,d;H, N)$ provided that $H \ge N^{1-2/d+\varepsilon}$ if $d=d_1=d_2$. For the unequal-degree case $d_1\neq d_2$, we establish a nontrivial bound of $B(d_1,d_2;H,N)$ provided that $H\ge N^{1-\frac{4}{d_1+d_2}+\varepsilon}$. If $d_1\ge d_2$, the smallest shift for which a nontrivial bound for $B(d_1,d_2;H,N)$ can be established is $H > N^{(d_2-1)/(d_2+1)}$, based on a result of Friedlander and Iwaniec \cite{FI05} (see below subsection \ref{subsec:Friedlander Iwaniec result}). The shifted convolution sums are typically analyzed using either the circle method or the spectral method (see \cite{Michel07}). Our approach employs a version of the circle method, in particular the delta method of Duke Friedlander Iwniec \cite{DFI93}. To the best of our knowledge, this is the first application of the delta method to shifted convolution sums involving higher-degree Fourier coefficients associated with the Hecke--Maass cusp form for $\mathrm{SL}(d,\mathbb{Z})$ with $d\ge 4$. Before stating our main theorems, we fix some notation.

\subsection{Notations}
Throughout the sequel, we denote by $\varepsilon$ a small positive constant, whose value may change from one occurrence to the next. The functions $V$, $U$, and $W$ will refer to compactly supported smooth weight functions whose definitions may change depending on the context. The notation $n \sim N$ means that $N < n \le 2N$. We write $X \asymp Y$ to mean $N^{-\varepsilon} < |X/Y| < N^{\varepsilon}$ whenever $Y \neq 0$. For quantities $X$ and $Y$ depending on a variable $x$, we write $X = O(Y)$ or $X \ll Y$ if $|X| \le CY$ for some constant $C$. 

\subsection{Main results}
Our first main result concerns the case in which $\pi_{1}$ and $\pi_{2}$ 
are Hecke--Maass cusp forms for $\mathrm{SL}(d,\mathbb{Z})$ of the same degree $d\ge 4$.
\begin{theorem}\label{Main Th 1}
Let $d\ge 4$. For any $\varepsilon>0$, we have 
$$
B(d,d;H,N)\ll N^{\varepsilon}(N^{d-1}{H^{-d}}),
$$
where the implied constant depends only on $\pi_1, \pi_2$ and $\varepsilon$. In particular, this bound yields a power-saving improvement over the trivial estimate provided that for $H \ge N^{1-2/d+\varepsilon}$.
\end{theorem}
We next consider the case in which $\pi_{i}$ are Hecke--Maass cusp forms for $\mathrm{SL}(d_{i},\mathbb{Z})$ for $i=1,2$, with $d_{1}> d_{2}$. The result can be derived from a result of Friedlander and Iwaniec (see below subsection~\ref{subsec:Friedlander Iwaniec result}) already implies a nontrivial bound of $B(d_1,d_2;H,N)$ provided that 
$H \;\ge\; N^{\frac{d_{2}-1}{\,d_{2}+1\,}+\varepsilon}.$
In what follows, we focus on the complementary regime 
$$
H \le N^{\frac{d_{2}-1}{\,d_{2}+1\,}}.
$$
Theorem~\ref{Main Th 1} generalizes to this setting with essentially the same strength for higher-rank Hecke--Maass forms of different degrees. The next theorem provides the corresponding statement in this more general framework.
\begin{theorem} \label{Main Th 2} 
Let $d_1,\, d_2\ge 4$ be such that $d_1>d_2$ and suppose that $H \le N^{(d_2-1)/(d_2+1)}$. Then for any $\varepsilon>0$, we have 
$$
B(d_1,d_2;H,N)\ll N^\varepsilon N^{\frac{(d_1+d_2-2)}{2}}H^{-\frac{(d_1+d_2)}{2}},
$$
where the implied constant depends only on $\pi_1,\, \pi_2$ and $\varepsilon$.  In particular, this bound yields a power-saving improvement over the trivial estimate provided that $H\ge N^{1-\frac{4}{d_1+d_2}+\varepsilon}$. 
\end{theorem}

\begin{remark}
The lower bound obtained for the shift $H$ in Theorem~\ref{Main Th 2} improves upon the bound $H \;\ge\; N^{\,\frac{d_{2}-1}{d_{2}+1}+\varepsilon}$ only in the case $d_1=d_2+1$.
\end{remark}

\subsection{A result due to Friedlander and Iwaniec}\label{subsec:Friedlander Iwaniec result} Without loss of generality, we assume that $d_1\ge d_2$. The expression in \eqref{def of B(d_1,d_2;H,N)} can be rewritten as 
\begin{equation*}
B(d_1,d_2;H,N)=\frac{1}{H}\sum\limits_{n\sim N} A_{\pi_1}(n) S(n,H), 
\end{equation*}
where 
$$
S(n,H)=\sum\limits_{m\le n+2H} A_{\pi_2}(m)-\sum\limits_{m\le n+H} A_{\pi_2}(m).
$$
Using the Cauchy-Schwarz inequality and then applying Lemma~\ref{lemma Rama Bound} (see below), one can derive the following estimate
\begin{equation} \label{after cauchy}
B(d_1,d_2;H,N)\ll \frac{N^{(1+\varepsilon)/2}}{H} \Big( \sum\limits_{n\sim N} |S(n,H)|^2\Big)^{1/2}.
\end{equation} 
By using  Lemma~\ref{FI sum bbdd} (see below), we get
\begin{equation}\label{S bdd}
|S(n,H)|\ll N^{\frac{d_2-1}{d_2+1}+\varepsilon}.
\end{equation}
Combining \eqref{after cauchy} and \eqref{S bdd}, and after renaming $\varepsilon$ if necessary, we get
\begin{equation} \label{Friedlander Iwaniec result}
B(d_1,d_2;H,N)\ll N^{\varepsilon}N^{1+\frac{d_2-1}{d_2+1}}H^{-1}
\end{equation}
Here, the implied constant depends only on $\pi$ and $\varepsilon$. 
In particular, the above bound yields a power-saving improvement over the trivial estimate of $B(d_1,d_2;H,N)$ provided that $H \ge N^{\frac{d_2-1}{d_2+1}+\varepsilon}$. 

\begin{remark}
In the above proof, we obtain a nontrivial power saving for $B(d_1,d_2;H,N)$, provided that the shift satisfies $H \ge N^{\frac{d_2-1}{d_2+1}+\varepsilon}$. The argument relied on obtaining cancellation over a single Fourier coefficient. In this article, we employ the delta method, which enables us to achieve cancellation over two Fourier coefficients. This additional cancellation allows us to derive a stronger result, improving the lower bound for the shift to $H\ge N^{1-\frac{4}{d_1+d_2}+\varepsilon}$.     
\end{remark}

\subsection{Moment of degree-$d$ $L$-function}
We consider the second moment problem 
\begin{equation*}
    M_f(T)=\int\limits_{T}^{2T}\big|L(1/2+it,f) \big|^2dt,
\end{equation*}
where $f$ is a Hecke-Maass cusp form for $\mathrm{SL}(d,\mathbb{Z})$. 
The generalized Lindel\"of hypothesis proposes that 
$$
M_f(T) \ll T^{1+\varepsilon}.
$$
To obtain the trivial bound  $M_f(T) \ll T^{d/2 + \varepsilon}$, we first apply the approximate functional equation (see Theorem~5.3, \cite{IwaKowBook}), which allows us to truncate the sum over $n$ in the integrand at $N = T^{d/2 + \varepsilon}$. Expanding the absolute square and evaluating the resulting oscillatory integral, we obtain essentially the following.
\begin{equation} \label{Mf bdd}
M_f(T)\ll T^{1-d/2} \Big| \sum\limits_{h \sim T^{d/2-1}} \,\, \sum\limits_{n\sim T^{d/2}} A_\pi(n)A_\pi(n+h)\Big|.
\end{equation} 
By applying the Cauchy--Schwarz inequality together with the Ramanujan bound on average (see below Lemma~\ref{lemma Rama Bound}), the double sum can be bounded above by $T^{d-1}$. Consequently, we obtain the upper bound 
$$ 
M_f(T) \ll T^{d/2+\varepsilon}.
$$
Hence, we obtain 
\begin{equation*}
\sum\limits_{n\sim T^{d/2}} \frac{A_\pi(n)}{n^{1/2+it}}\ll T^{d/4+\varepsilon}.
\end{equation*}
This implies that
$$
L(1/2+it,f)\ll t^{d/4+\varepsilon}.
$$
\begin{remark}
In the above subsection, we obtain the convexity bound for high-degree $L$-functions. On the other hand, if Theorem~\ref{Main Th 1} holds for $H \geq N^{1 - \frac{2}{d} - \eta}$ for some positive $\eta$, then applying it to \eqref{Mf bdd}, yields a subconvexity bound for the corresponding $L$-function. Thus, we are precisely at the boundary case, and any further improvement of our result would yield a subconvexity bound for the automorphic $L$-function in the $t$-aspect through the approach based on the upper bound of the second moment of automorphic $L$-functions. This method has its origin in the classical ideas of Hardy and Littlewood \cite{HardyLittle1916}, whose pioneering work laid the foundation for the moment approach to subconvexity problems. Despite substantial progress over the decades, deriving subconvexity purely from moment estimates remains a deep and largely unresolved problem in higher rank settings $\mathrm{GL}(d)$ for $d\ge 4$. In spectacular recent work,
 Nelson \cite{Paul2021} claimed a remarkable general subconvexity bound for $\mathrm{GL}(d)$ automorphic $L$-functions in the $t$-aspect.
\end{remark}

\subsection{The description of our method}
Let us now briefly describe our method. To separate the oscillations appearing in the sum $B(d_1,d_2;H,N)$ in \eqref{B(d_1,d_2;H,N) sum before delta}, we apply the delta method, which yields the expression in \eqref{After Application of delta}. A trivial estimation gives $B(d_1,d_2;H,N)\ll N^{2+\varepsilon}$, so we need to save $N$ (and slightly more).
Applying the Poisson summation formula to the $h$-sum yields a saving of $\frac{H}{N^{1/2}}$. A small trick used in \eqref{B sum dual} forces $q=1$, which allows us to save the full modulus $Q = N^{1/2}$. Thus, the total saving at this stage is
$$
N^{1/2} \cdot \frac{H}{N^{1/2}} = H.
$$

To prove Theorem~\ref{Main Th 1}, we consider the case $d = d_1 = d_2$. After applying the functional equation to both the $n$-sum and the $m$-sum in \eqref{B sum q=1}, we obtain the saving
$$
\frac{N}{(N/H)^{d/2}} \cdot \frac{M}{(M/H)^{d/2}}
\asymp \frac{N^{2}}{(N/H)^{d}}.
$$
Next, by analyzing the $x$-integral appearing in the sum \eqref{Th1: B(d_1,d_2;H,N) with x}, we save $N/H$. Altogether, the total saving is
$$
H \cdot \frac{N^{2}}{(N/H)^{d}} \cdot \frac{N}{H}
= N^{3-d} H^{d}.
$$
Hence, we obtain a nontrivial power-saving bound of $B(d_1,d_2;H,N)$ provided that the total saving satisfies
$$N^{3-d} H^{d} \ge N^{1+\varepsilon}.
$$
This is equivalent to the condition $$H \ge N^{1 - 2/d + \varepsilon}.$$

To prove Theorem~\ref{Main Th 2}, we consider the unequal case $d_1 > d_2$. After applying the functional equation to both the $n$-sum and the $m$-sum in \eqref{B sum q=1}, we obtain the saving
$$
\frac{N}{(N/H)^{d_1/2}} \cdot \frac{M}{(M/H)^{d_2/2}} \asymp \frac{N^2}{(N/H)^{\frac{d_1+d_2}{2}}}.
$$
Next, by estimating the $x$-integral appearing in the sum \eqref{2. B(d_1,d_2;H,N) voronoi}, we gain a further saving $(N/H)^{1/2}$ at \eqref{2.B sum last}. To get the strength of Theorem~\ref{Main Th 1}, we need to save $(N/H)^{1/2}$. To this end, we apply the Cauchy--Schwarz inequality (see \eqref{eq with omega}), followed by the Poisson summation formula on the $n$-sum, which yields \eqref{after poisson Sigma}. We note that
$$
\text{dual length}
= \frac{\text{conductor}}{\text{initial}}
= \frac{N/H}{\Tilde{N}}.
$$
The natural choice \eqref{Th2: condition: H/N} forces $\Tilde{N} > N/H$, and hence only the zero frequency contributes. Moreover, since $\Tilde{M} > N/H$, we obtain the saving
$$
\bigl( \min\{ \tilde{M},\, N/H \} \bigr)^{1/2} = (N/H)^{1/2}.
$$
Altogether, the total saving is
$$
H \cdot \frac{N^2}{(N/H)^{\frac{d_1+d_2}{2}}} \cdot
(N/H)^{1/2} \cdot
(N/H)^{1/2}
=
N^{3 - \frac{d_1 + d_2}{2}}\, H^{\frac{d_1 + d_2}{2}}.
$$
Hence we obtain a nontrivial power-saving bound of $B(d_1,d_2;H,N)$ provided that the total saving satisfies
$$
N^{3 - \frac{d_1 + d_2}{2}}\, H^{\frac{d_1 + d_2}{2}}
\ge N^{1+\varepsilon}.
$$
This is equivalent to the condition
$$
H \ge N^{1 - \frac{4}{d_1 + d_2} + \varepsilon}.
$$

\section{Preliminaries}
In this section, we begin with some basic properties of $\mathrm{SL}(d,\mathbb{Z})$ automorphic forms. For background material on Hecke--Maass cusp forms, we refer the reader to \cite{Gold06b}.
\subsection{Hecke–Maass cusp forms for $\mathrm{SL}(d, \mathbb{Z})$}
Let $\pi$ be a Hecke--Maass cusp form for $\mathrm{SL}(d,\mathbb{Z})$ with the spectral parameter $(\alpha_1, \dots, \alpha_d)\in \mathbb{C}^{d}$. Let $A_\pi$ be the Fourier-Whittaker coefficient of $\pi$. The conjugate of $A$ is defined as $\overline{A_\pi(n_1,\dots, n_{d-1})}= A_\pi(n_{d-1},\dots, n_{1})$. We also assume $\pi$ is a Hecke eigenform with a  Fourier coefficient $A_\pi(1, \dots, 1)$ to be $1$. The Dual Maass form of $\pi$ is denoted by $\tilde{\pi}$. Let $A_{\tilde{\pi}}(n_1,\dots, n_d)$ be the Fourier-Whittaker coefficients of $\tilde{\pi}$, then 
$$
A_{\tilde{\pi}}(n_1,\dots,n_{d-1})=\overline{A_{\pi}(n_1,\dots,n_{d-1})}=A_{\pi}(n_{d-1},\dots,n_1).
$$
Throughout the sequel, for simplicity of notation, we write $A_\pi(n)$ in place of $A_{\pi}(1,\ldots,1,n)$ and $A_{\tilde{\pi}}(n)$ in place of $A_{\tilde{\pi}}(n,1,\ldots,1)$. 
We now define 
\begin{equation} \label{DEF. of gamma,0,1} 
\gamma_{\delta_0} (s):=i^{-d\delta_0}\pi^{-d(1/2-s)}\prod_{j=1}^d 
\frac{\Gamma\!\left(\frac{1- s +\delta_0 - \overline{\alpha_j}}{2}\right)}
{\Gamma\!\left(\frac{s+\delta_0-\alpha_j}{2}\right)}.
\end{equation}
For even Maass forms we define $\delta_0=0$ and for odd Maass forms we define $\delta_0=1$. 
We now define 
\begin{equation*}
\Omega(y):=\frac{1}{2\pi i} \int\limits_{\Re (s)=\sigma} (y)^s \Tilde{\omega}(s) \gamma_{\delta_0}(s) \, ds,
\end{equation*}
where $\sigma<1+\max\limits_{1\le j\le d} \{\Re(\alpha_j)\}$ and  $\Tilde{\omega}(s)=\int\limits_{0}^{\infty}\omega(y)y^{s-1}dy$ is the Mellin transform of $\omega$. Note that the Luo--Rudnick--Sarnak bound says that
$$
\Re(\alpha_j) \le \frac{1}{2}-\frac{1}{d^2+1},
$$
see \cite[Theorem 12.5.1]{Gold06b}.
The standard $L$-function $\pi$ is define by
$$
L(s,\pi)=\sum\limits_{n=1}^{\infty}\frac{A_\pi(n)}{n^s}.
$$
This $L$-function satisfies the functional equation
\begin{equation}\label{functional equ}
L(s,\pi)=\gamma_{\delta_0}(s)L(1-s,\tilde{\pi}).
\end{equation}
The following bound is well-known and follows from standard properties of the Rankin--Selberg $L$-function. 
\begin{lemma}[Ramanujan conjecture on average] \label{lemma Rama Bound}
For any $\varepsilon > 0$, we have
$$
\sum_{n \leq X} \left| A_\pi(n) \right|^2 \;\ll_{\pi,\varepsilon}\; X^{1+\varepsilon},
$$
where the implied constant depends on $\pi$ and $\varepsilon$.
\end{lemma}
The following bound is derived from the result of Friedlander and Iwaniec \cite[proposition 1.1]{FI05}. 
\begin{lemma}\label{FI sum bbdd}
We have 
$$
 \sum\limits_{n\le X} A_\pi(n)\ll_{\pi,\varepsilon} X^{\frac{d-1}{d+1}+\varepsilon},
$$
where the implied constant depends on $\pi$ and $\varepsilon$.
\end{lemma}

\subsection{The stationary phase method}
In this subsection, we discuss the method of the stationary phase for evaluating oscillatory integrals of the form
$$
I:=\int_{a}^{b} w(x)e(f(x)) dx,
$$
where $f$ and $w$ are smooth real valued functions $[a,b]$. We begin with a lemma, for the case when stationary points do not exist, which we get by repeated integration by parts, showing that the oscillating integral is negligibly small.
\begin{lemma}\label{Stationary1}
Under the smoothness assumptions on $f$ and $w$, we obtain
\begin{equation*}
I\ll \frac{\mbox{Var}(w)}{\min |f^{(j)}(x)|^{1/j}},
\end{equation*}
where $\mbox{Var}$ is the total variance of $w$ on $[a,b]$. Furthermore, let $f'(x)\ge X$ and $f^{(j)}(x)\ll X^{1+\varepsilon}$ for $j\ge 2$ with $\mbox{Supp}(w)\subset (a,b) $ and $w^{(j)}(x)\ll_{a,b,j} 1$. Then we have 
$$
I\ll_{a,b,j,\varepsilon} X^{-j+\varepsilon}.
$$
\end{lemma}
The next lemma provides an asymptotic estimate for the integral when a unique stationary point exists. It turns out that only a small neighbourhood around the stationary point contributes significantly to the value of the integral.
\begin{lemma} \label{Stationary2}
Let $0<\eta<1/10$, $X,Y, U_0,U_1 , R>0$, and $Z:=R+X+Y+U_1+1$ and assume that
\begin{equation}\label{lemma cond1}
    Y\ge Z^{3\eta}, \,\, U_1\ge U_0\ge \frac{RZ^{\eta/2}}{\sqrt{Y}}.
\end{equation}
Suppose that $w$ is a smooth function on $\mathbb{R}$ with support on an interval $J$ of length $U_1$, satisfying 
$$
w^{(j)}(x)\ll_j XU_0^{-j} \mbox{  for all } j=0,1,2,\dots.
$$
Further assume that $f$ is a smooth function on $J$ such that there exists a unique $x_0\in J$ such that $f'(x_0)=0$ and for all $x$
$$
f''(x)\gg YR^{-2}, \,\,\, f^{(j)}(x)\ll_j YR^{-j} \mbox{ for } j=0,1,2,\dots.
$$
Then  the integral defined by
$$
I:=\int\limits_{\mathbb{R}} w(x)e(f(x))dx
$$
has an asymptotic of the form
\begin{equation} \label{lemma asymp}
I=\frac{e(f(x_0))}{\sqrt{f''(x_0)}} \sum\limits_{n\le 3\eta^{-1} A} g_n(x_0)+O_{A,\eta}(Z^{-A}),
\end{equation}
where 
$$
g_n(x_0)=\frac{\sqrt{2\pi} e^{\pi i/4}}{n!}\Big(\frac{i}{2f''(x_0)}\Big)^n F^{(2n)}(x_0),
 \mbox{ and }
F(x)=w(x)e^{i(f(x)-f(x_0)-1/2f''(x_0)(x-x_0)^2)}
$$
Furthermore, every $g_n$ is a rational functions in $f'', f''', \dots $ satisfying the derivative bound 
\begin{equation}\label{lemma deri bdd}
\frac{d^j}{dt_0^j}g_n(x_0)\ll_{j,n} X(U_0^{-j}+ R^{-j})((U_0^2Y/R^2)^{-n}+Y^{-n/3}). 
\end{equation}
\end{lemma}
\begin{proof}
For the proof, we refer to \cite{BlKhYo2013}, Lemma 8.2.
\end{proof}
\begin{remark}
As observed in \cite[p.~2639]{BlKhYo2013}, the condition \eqref{lemma cond1} together with the derivative bound \eqref{lemma deri bdd} implies that, in the asymptotic expansion \eqref{lemma asymp}, each successive term is smaller than the preceding one. Hence, it suffices to consider only the leading term in the asymptotic expansion, provided that the condition \eqref{lemma cond1} is verified.
\end{remark}

\subsection{The delta method}
Let us briefly recall a version of the delta method due to Duke, Friedlander, and Iwaniec \cite{DFI93}. More specifically, we will use the expansion (20.157) given in Chapter 20 of \cite{IwaKowBook}. Let $\delta: \mathbb{Z} \longrightarrow \{0, 1\}$ be the Kronecker delta function defined by
$$
\delta(n,m)=
\begin{cases}
	1 &\text{if}\,\,n= m, \\
	0 &\text{otherwise.}
\end{cases}
$$
We seek a Fourier expansion which matches with $\delta$ in the range $[-2Q,2Q]$. To this end, let $Q \ge 2$ and take any $n,m \in \mathbb{Z} \cap [-2Q, 2Q]$. Then we have
\begin{equation*}
\delta(n,m)=\frac{1}{Q}\sum\limits_{1\le q\le Q}\frac{1}{q} \, \sideset{}{^\ast}\sum_{a\bmod q} e\Big(\frac{(n-m)a}{q}\Big) \int\limits_{\mathbb{R}}g(q,x) e\Big(\frac{(n-m)x}{qQ}\Big) dx,
\end{equation*}
where $g(q,x)$ satisfies the following properties
\begin{equation*}
\begin{split}
g(q,x)&=1+O\Big(\frac{Q}{q}\Big(\frac{q}{Q}+|x|\Big)^A\Big), \quad g(q,x)\ll |x|^{-A}, \\
&x^{j}\frac{\partial^j}{\partial{x}^j}g(q,x)\ll \min (Q/q,|x|^{-1}) \log Q
\end{split}
\end{equation*}
for any $A>1$, integer $j\ge 1$. Here, as later,  the $\ast$ attached to the sum symbol indicates that $a$ and $q$ are co-prime, and $e(x)=e^{2\pi i x}$. Moreover, the second property of $g(q,x)$ implies that the effective range of the integration over $x$ is $[-Q^\varepsilon, Q^\varepsilon].$ It follows that if $q\ll Q^{1-\varepsilon}$ and $x\ll Q^{-\varepsilon}$, then $g(q,x)$ can be replaced by $1$ at the cost of a negligible error term. In the complementary range, using the third property of $g(q,x)$, we have 
$$
x^{j}\frac{\partial^j}{\partial{x}^j}g(q,x)\ll Q^\varepsilon. 
$$
Finally in \cite{Munshi22}, by Parseval and Cauchy we get
$$
\int \Big(|g(q,x)|+|g(q,x)|^2 \Big) dx \ll Q^\varepsilon
$$
i.e., $g(q,x)$ has average size $1$ in the $L^1$ and $L^2$ sense. We summarize the above observations in the following lemma.
\begin{lemma} \label{lemma Delta}
Under the above notations, we have
$$
\delta(n,m)=\frac{1}{Q}\sum\limits_{1\le q\le Q}\frac{1}{q} \, \sideset{}{^\ast}\sum_{a\bmod q} e\Big(\frac{(n-m)a}{q}\Big) \int\limits_{\mathbb{R}} B_0(x)g(q,x) e\Big(\frac{(n-m)x}{qQ}\Big) dx +O(Q^{-2\cdot2025}),
$$
where $B_0$ is a smooth bump function supported in $[-2Q^{\varepsilon}, 2Q^{\varepsilon}]$ satisfying $B_0(x)=1$ for $x\in [-Q^{\varepsilon}, Q^{\varepsilon}]$ and $B_0^{(j)}\ll 1$.
\end{lemma}
\begin{proof}
See Chapter~20 of \cite{IwaKowBook} and Lemma~15 of \cite{Huang2021}.
\end{proof}

\section{Treatment of $B(d_1,d_2;H,N)$}
We rewrite the sum $B(d_1,d_2;H,N)$ appearing in \eqref{def of B(d_1,d_2;H,N)} in the smoothed form
\begin{equation*}
B(d_1,d_2;H,N)=\frac{1}{H}\sum\limits_{h=1}^{\infty}W\Big(\frac{h}{H}\Big)\sum_{n=1}^{\infty}A_{\pi_1}(n) A_{\pi_2}(n+h)V\Big(\frac{n}{N}\Big),
\end{equation*}
where $W$ and $V$ are compactly supported smooth functions on the interval $[1,2]$ satisfying the condition $W^{(j)}\ll_j 1$ and $V^{(j)}\ll_j 1$. Now we detect the equation $n+h=m$ by the delta symbol and rewrite our main sum as
\begin{equation} \label{B(d_1,d_2;H,N) sum before delta}
B(d_1,d_2;H,N)=\frac{1}{H}\sum\limits_{h=1}^{\infty}W\Big(\frac{h}{H}\Big)\sum_{n=1}^{\infty}A_{\pi_1}(n) V\Big(\frac{n}{N}\Big) \sum_{m=1}^{\infty} A_{\pi_2}(m) U\Big(\frac{m}{M}\Big) \delta(m,n+h),
\end{equation}
where $M=N+h\asymp N$, and $U$ is a compactly supported smooth function in the interval $[1/2,5/2]$ satisfying $U(x)=1$ for $[1,2]$ and $U^{(j)}\ll_j 1$. To separate the oscillations involved in the sum $B(d_1,d_2;H,N)$, we apply the delta method expansion (Lemma~\ref{lemma Delta}) with the choice $Q = N^{1/2}$, we obtain
\subsection{Applying the delta method}
Applying the delta method, we obtain 
\begin{equation} \label{After Application of delta}
\begin{split}
B(d_1,d_2;H,N)= &\frac{1}{HQ} \sum\limits_{1\le q\le Q}\frac{1}{q} \,\,\sideset{}{^\ast}\sum_{a\bmod q}\,\int\limits_{\mathbb{R}} B_0(x)g(q,x) \sum\limits_{h\in \mathbb{Z}}W\left(\frac{h}{H}\right) e\left(-\frac{ah}{q}\right) e\left(-\frac{xh}{qQ}\right)\\
&\times \sum\limits_{n=1}^{\infty}A_{\pi_1}(n) V\left(\frac{n}{N}\right)e\left(-\frac{an}{q}\right) e\left(-\frac{nx}{qQ}\right) \\
&\times \sum\limits_{m=1}^{\infty}A_{\pi_2}(m) U\left(\frac{m}{M}\right)e\left(\frac{am}{q}\right) e\left(\frac{mx}{qQ}\right) dx +O(N^{-2025}).
\end{split}
\end{equation}

\subsection{Applying the Poisson summation formula}
Changing the variable $h=\beta +hq$ in the $h$-sum of the above sum, we obtain
\begin{equation*}
\sum\limits_{h \in \mathbb{Z}} W\left(\frac{h}{H}\right) e\left(-\frac{ah}{q}\right) e\left(-\frac{xh}{qQ}\right)= \sum\limits_{h \in \mathbb{Z}}\sum\limits_{\beta \bmod q} W\left(\frac{\beta+ qh}{H}\right) e\left(-\frac{a\beta}{q}\right) e\left(-\frac{x(\beta+qh)}{qQ}\right).
\end{equation*}
Applying the Poisson summation formula on $h$-sum, we have  
$$ \sum\limits_{\beta \bmod q}e\left(-\frac{a\beta}{q}\right) \sum\limits_{h \in \mathbb{Z}} \int\limits_{\mathbb{R}} W\left(\frac{\beta+ qy}{H}\right) e\left(-\frac{x(\beta+qy)}{qQ}\right) e(-yh)dy.
$$
Putting $\frac{\beta+qy}{H}=z$, the above sum becomes
\begin{equation} \label{Gen: After Poissonon on h}
\begin{split}
&\frac{H}{q}\sum\limits_{\beta \bmod q}e\left(-\frac{a\beta}{q}\right) \sum\limits_{h \in \mathbb{Z}} \int\limits_{\mathbb{R}} W(z) e\left(-\frac{xzH}{qQ}\right) e\left(\frac{(\beta-Hz)h}{q}\right)dz\\
&=\frac{H}{q} \sum\limits_{\beta \bmod q}e\left(-\frac{a\beta}{q}\right) \sum\limits_{h \in \mathbb{Z}} \Hat{W}\left(\frac{xH}{qQ}+\frac{Hh}{q}\right) e\left(\frac{\beta h}{q}\right)\\
&=\frac{H}{q} \sum\limits_{\beta \bmod q} \sum\limits_{h \in \mathbb{Z}} \Hat{W}\left(\frac{xH}{qQ}+\frac{Hh}{q}\right) e\left(\frac{\beta (h-a)}{q}\right).
\end{split}
\end{equation}
Combining \eqref{After Application of delta} and \eqref{Gen: After Poissonon on h}, we have 
\begin{equation}\label{B sum dual}
\begin{split}
B(d_1,d_2;H,N)= &\frac{1}{Q} \sum\limits_{1\le q\le Q} \frac{1}{q^2}\,\,\sideset{}{^\ast}\sum_{a\bmod q}\,\int\limits_{\mathbb{R}} B_0(x)g(q,x) \sum\limits_{\beta \bmod q} \sum\limits_{h \in \mathbb{Z}} \Hat{W}\left(\frac{xH}{qQ}+\frac{Hh}{q}\right) e\left(\frac{\beta (h-a)}{q}\right)\\
&\times \sum\limits_{n=1}^{\infty}A_{\pi_1}(n) V\left(\frac{n}{N}\right)e\left(-\frac{an}{q}\right) e\left(-\frac{nx}{qQ}\right) \\
&\times \sum\limits_{m=1}^{\infty}A_{\pi_2}(m) U\left(\frac{m}{M}\right)e\left(\frac{am}{q}\right) e\left(\frac{mx}{qQ}\right) dx +O(N^{-2025}).
\end{split}
\end{equation}

\subsection{The zero frequency} 
Suppose we take the length of the shift to satisfy $H \ge N^{1/2+\varepsilon}$ for some $\varepsilon>0$ (see Remark~\ref{condition on H}). Since $W$ is a compactly supported smooth function on $[1,2]$ with $W^{(j)} \ll_j 1$, its Fourier transform $\hat{W}(\xi)$ is rapidly decaying for $|\xi|\gg 1$ and is negligible unless $|\xi|\ll 1$. Hence the expression $\frac{xH}{qQ}+\frac{Hh}{q}$, appearing inside $\hat{W}$, must be very small for the corresponding term to contribute significantly to the above equation \eqref{B sum dual}. Since $H \ge N^{1/2+\varepsilon}$ and $q \asymp N^{1/2}$, we have $\frac{Hh}{q} \gg N^{\varepsilon}$ for every $h \ne 0$, so all nonzero $h$ lie outside the effective support of $\hat{W}$. Therefore only the zero frequency $h=0$ contributes significantly. Substituting $h=0$ into the sum \eqref{B sum dual}, we get
\begin{equation*}\label{B sum zero} 
\begin{split}
B(d_1,d_2;H,N)= &\frac{1}{Q} \sum\limits_{1\le q\le Q} \frac{1}{q^2}\,\,\sideset{}{^\ast}\sum_{a\bmod q}\,\int\limits_{\mathbb{R}} B_0(x)g(q,x) \Hat{W}\left(\frac{xH}{qQ}\right) \sum\limits_{\beta \bmod q}  e\left(\frac{-\beta a}{q}\right)\\
&\times \sum\limits_{n=1}^{\infty}A_{\pi_1}(n) V\left(\frac{n}{N}\right)e\left(-\frac{an}{q}\right) e\left(-\frac{nx}{qQ}\right) \\
&\times \sum\limits_{m=1}^{\infty}A_{\pi_2}(m) U\left(\frac{m}{M}\right)e\left(\frac{am}{q}\right) e\left(\frac{mx}{qQ}\right) dx +O(N^{-2025}).
\end{split}
\end{equation*}
Since $(a,q)=1$, then 
$$
\sum\limits_{\beta \bmod q}  e\left(\frac{-\beta a}{q}\right)= 0
$$ unless $q=1$. Hence 
\begin{equation}\label{B sum q=1}
\begin{split}
B(d_1,d_2;H,N)= &\frac{1}{Q} \int\limits_{\mathbb{R}} B_0(x)g(1,x) \Hat{W}\left(\frac{xH}{Q}\right) 
\sum\limits_{n=1}^{\infty}A_{\pi_1}(n) V\left(\frac{n}{N}\right)e\left(-\frac{nx}{Q}\right) \\
&\times \sum\limits_{m=1}^{\infty}A_{\pi_2}(m) U\left(\frac{m}{M}\right) e\left(\frac{mx}{Q}\right) dx +O(N^{-2025}).
\end{split}
\end{equation} 

\begin{remark}\label{condition on H}
The assumption $H \ge N^{1/2+\varepsilon}$ can be imposed without loss of generality. 
Indeed, our main theorems provide lower bounds for the shift parameter $H$ of the form 
$H \ge N^{1 - \frac{2}{d}+\varepsilon}$ in Theorem~\ref{Main Th 1} and 
$H \ge N^{1 - \frac{4}{d_1 + d_2}+\varepsilon}$ in Theorem~\ref{Main Th 2}. 
These inequalities imply $H \ge N^{1/2+\varepsilon}$ whenever $d \ge 4$ and 
$d_1 + d_2 \ge 8$, respectively, which correspond to the ranges relevant to our analysis. Hence, the condition $H\ge N^{1/2+\varepsilon}$ imposes no additional restriction and serves only to simplify the subsequent analysis.
\end{remark}

\section{Application of the functional equation}

We observe that the inner sums appearing in \eqref{B sum q=1} are of a similar form. In this section, we therefore focus on estimating
\begin{equation} \label{Def Xi}
\Xi(x):=\sum\limits_{n=1}^{\infty}A_\pi(n) V\left(\frac{n}{N}\right)e\left(-\frac{nx}{Q}\right).    
\end{equation}
For convenience, set
\begin{equation*} \label{Def w_x}
\omega_x(z):=V(z)e\left(-\frac{Nxz}{Q}\right).  
\end{equation*}
By applying the inverse Mellin transform, we obtain
\begin{equation} \label{after mellin Xi}
\begin{split}
\Xi(x)=\sum\limits_{n=1}^{\infty}A_\pi(n) \omega_x(n/N) =& \frac{1}{2\pi i} \int\limits_{\Re (s)=\sigma} N^s\, \Tilde{\omega}_x(s) \sum\limits_{n=1}^{\infty} \frac{A_{\pi}(n)}{n^s} ds\\
=& \frac{1}{2\pi i} \int\limits_{\Re (s)=\sigma} N^s\, \Tilde{\omega}_x(s) L(s,\pi)ds
\end{split} 
\end{equation}  
for any $\sigma$. Applying the functional equation to \eqref{after mellin Xi} for $L(s,\pi)$ (see \eqref{functional equ}), we have 
\begin{equation*}
\Xi(x)= \frac{1}{2 \pi i} \int\limits_{\Re (s)=\sigma} N^s \Tilde{\omega}_x(s) \gamma_{\delta_0}(s)  L(1-s,\tilde{\pi}) ds.
\end{equation*}

Let $\mathcal{V} = \{(V_0, \Tilde{N})\}$ be a smooth dyadic partition of unity, consisting of pairs $(V_0, \Tilde{N})$, 
where $V_0 : [1,2] \to \mathbb{R}_{\ge 0}$ is a smooth function satisfying 
$$
\sum_{(V_0,\, \tilde{N})} V_0\!\left(\frac{n}{\tilde{N}}\right) = 1, 
\quad \text{for all } n \in (0, \infty).
$$
Moreover, the collection is locally finite in the sense that for any given $\ell \in \mathbb{Z}$, 
there are only finitely many pairs with 
$
\Tilde{N} \in [2^{\ell}, 2^{\ell+1}].$ Moving the line of inetegration to $\sigma=-\epsilon$ and expanding the $L$-function into its automorphic $L$-series, and using a smooth dyadic partition of unity $\mathcal{V} $, we transform the sum $\Xi(x)$ into
\begin{equation} \label{without cut off n}
\Xi(x)=\sum\limits_{\mathcal{V}}\sum\limits_{n=1}^{\infty} \frac{A_{\tilde{\pi}}(n)}{n} V_0\Big(\frac{n}{\Tilde{N}}\Big) \Omega_x(nN), 
\end{equation}
where 
\begin{equation}\label{Omega_x nN}
\Omega_x(nN)=\frac{1}{2\pi i}\int\limits_{\Re (s)=-\epsilon} (Nn)^s \Tilde{\omega}_x(s) \gamma_{\delta_0}(s) \, ds.  
\end{equation}

Recalling $\omega_x(s)$, and putting $s=\sigma+i\tau$, we get
\begin{equation} \label{omega_x2}
\begin{split}
\Tilde{\omega}_x(\sigma+i\tau)&=\int\limits_{0}^{\infty}V(z)e\Big(-\frac{xzN}{Q} + \frac{\tau}{2\pi}\log z\Big)z^{\sigma-1}dz. 
\end{split}
\end{equation}
We define the phase function
$$
g(z) := -\frac{N x z}{Q} + \frac{\tau}{2\pi} \log z.
$$
Then
$$
g'(z)=-\frac{Nx}{Q}+\frac{\tau}{2\pi z} \mbox{ and } g''(z)=-\frac{\tau}{2\pi z^2}.
$$
The stationary point is 
\begin{equation*}
 z_0 = \frac{\tau Q}{2\pi N x}.   
\end{equation*}
Applying repeated integration by parts, the integral $\Tilde{\omega}_x(\sigma+i\tau)$
is negligible unless
\begin{equation} \label{st point}
  |\tau| \asymp \frac{N |x|}{Q} \,\, \mbox{ and } \,\, \text{sgn}(\tau)=\text{sgn}(x).  
\end{equation}
Using the Stirling formula, for fixed $\sigma$ we have
$$
\gamma_{\delta_0}(\sigma+i\tau)\ll_{\sigma,\pi}(1+|\tau|)^{d(-\sigma+1/2)}.
$$
Plugging the above bound into the equation \eqref{Omega_x nN} and using \eqref{st point}, we get
\begin{equation*}
\begin{split}
\Omega_x(Nn) &=\frac{1}{2\pi}\int\limits_{-\infty}^{\infty} (Nn)^{\sigma+i\tau} \Tilde{\omega}_x(\sigma+i\tau) \gamma_{\delta_0}(\sigma+i\tau) \, d\tau  \\
&\ll (Nn)^\sigma \Big(\frac{N|x|}{Q}\Big)^{-d\sigma+ d/2+1}.
\end{split}
\end{equation*}
Using $x\asymp Q/H$, we get
\begin{equation*}
\Omega_x(Nn)\ll (Nn)^\sigma \Big(\frac{N}{H}\Big)^{-d\sigma+ d/2+1} = (N/H)^{d/2+1} \Bigg(\frac{n}{\frac{(N/H)^d}{N}}\Bigg)^{\sigma}.
\end{equation*}
If $\tilde{N} \gg \frac{(N/H)^d N^{\varepsilon}}{N}$, we shift the contour sufficiently far to the left. On the other hand, if $\frac{(N/H)^d N^{-\varepsilon}}{N} \ll \tilde{N}$, we shift the contour to the right, keeping away from the poles of the gamma factors. This is possible since the contour satisfies $\tau \asymp N/H \gg 1$, while the poles lie on the positive real axis. Hence, we observe that the contribution $\Omega_x(Nn)$ from the above mentioned ranges is negligibly small. 

Let $\mathcal{V}^\ast$ be the subset of $\mathcal{V}$ consisting of those pairs $(V_0, \tilde{N})$ with 
$\tilde{N}$ restricted to the range 
\begin{equation} \label{tilde range N}
\frac{(N/H)^d N^{-\varepsilon}}{N} \ll \tilde{N} \ll \frac{(N/H)^d N^{\varepsilon}}{N}.
\end{equation}
Then, from \eqref{without cut off n}, we get
\begin{equation}\label{cut of n}
\Xi(x)=\sum\limits_{\mathcal{V}^\ast}\sum\limits_{n=1}^{\infty} \frac{A_{\tilde{\pi}}(n)}{n} V_0\Big(\frac{n}{\Tilde{N}}\Big) \Omega_x(nN) +O(N^{-2025}).
\end{equation}

\subsection{Simplification of the integrals}
We first simplify the integral $\omega_x(s)$. 
The condition \eqref{st point} ensures the existence of a stationary phase point within the support of the smooth function $V(z)$. Hence by Lemma~\ref{Stationary2} with $X\asymp U_0\asymp U_1\asymp R\asymp1,$ and $Y\asymp|\tau|$, the expression in \eqref{omega_x2} becomes
\begin{equation*} 
\begin{split}
\Tilde{\omega}_x(\sigma+i\tau) \asymp \frac{e\Big(-\frac{\tau}{2\pi}+ \frac{\tau}{2\pi}\log(\frac{\tau Q}{2\pi Nx})\Big) V\big(\frac{\tau Q}{2\pi N x}\big)}{\sqrt{\tau}} +O(N^{-2025}),
\end{split}
\end{equation*} 
where $V$ is a new compactly supported smooth function with bounded derivatives.
Plugging it into the \eqref{Omega_x nN}, we get 
\begin{equation*} 
\begin{split}
\Omega_x(nN)&\asymp\int\limits_{\Re (s)=\sigma} (Nn)^{\sigma+i\tau}\cdot \frac{e\Big(-\frac{\tau}{2\pi}+ \frac{\tau}{2\pi}\log(\frac{ \tau Q}{2\pi Nx})\Big)}{\sqrt{\tau}} \, \gamma_{\delta_0}(s) \, V\big(\frac{\tau Q}{2\pi N x}\big) ds +O(N^{-2025}),
\end{split}
\end{equation*}
where $V$ is a new compactly supported smooth function with bounded derivatives.

Now, we take $\delta_{0}=0$, corresponding to the case of even Hecke--Maass cusp forms. For odd Hecke--Maass cusp forms, i.e.~$\delta_{0}=1$, the arguments given below remain valid. Putting $\delta_{0}=0$, we get
\begin{equation*}
\Omega_x(nN)\asymp\int\limits_{\Re (s)=\sigma} (Nn)^{\sigma+i\tau}\cdot \frac{e\Big(-\frac{\tau}{2\pi}+ \frac{\tau}{2\pi}\log(\frac{ \tau Q}{2\pi Nx})\Big) V\big(\frac{\tau Q}{2\pi N x}\big) }{\sqrt{\tau}} \, \gamma_{0}(s) \, ds +O(N^{-2025}).
\end{equation*}
To cancel out the oscillation of gamma factors, we move the contour to $\sigma =1/2$. Then we have
\begin{equation} \label{Omega half}
\Omega_x(nN)\asymp\int\limits_{\mathbb{R}} (Nn)^{1/2+i\tau} \cdot \frac{e\Big(-\frac{\tau}{2\pi}+ \frac{\tau}{2\pi}\log(\frac{\tau Q}{2\pi Nx})\Big) V\big(\frac{\tau Q}{2\pi N x}\big)}{\sqrt{\tau}} \, \gamma_{0}(1/2+i\tau) \, d\tau  +O(N^{-2025}).
\end{equation}
The Stirling formula gives us 
\begin{equation*}
\Gamma(\sigma+i\tau)=\sqrt{2\pi} |\tau|^{\sigma-1/2+i\tau} \exp\big(-|\tau|\frac{\pi}{2} -i\tau+ i\frac{\pi}{2} (\sigma-1/2) \mbox{sgn}(\tau)\big) \big(1+O(|\tau|^{-1})\big)    
\end{equation*}
as $|\tau| \to \infty$. Hence for  $|\tau|\in [T,2T]$ with large $T$, and $\alpha_j\ll 1$, we have
\begin{equation*} 
\frac{\Gamma(\frac{1/2-i\tau -\overline{\alpha_j}}{2})}{\Gamma(\frac{1/2+i\tau-\alpha_j}{2})}=\Big(\frac{|\tau|}{2e}\Big)^{-i\tau}(1+ O(T^{-1})).
\end{equation*} 
Recalling the definition of \eqref{DEF. of gamma,0,1}, we get
\begin{equation*}
\gamma_{0}(1/2+i\tau)=\pi^{-d(1/2-s)}\Big(\frac{|\tau|}{2e}\Big)^{-id\tau }(1+ O(T^{-1})).
\end{equation*} 
Pugging it into the \eqref{Omega half}, we get
\begin{equation*} 
\Omega_x(nN) \asymp (Nn)^{1/2}\int\limits_{\mathbb{R}} (Nn)^{i \tau} e\Big(-\frac{\tau}{2\pi} +\frac{\tau}{2\pi} \log\big(\frac{Q\tau}{2\pi Nx} \big)\Big)\pi^{id\tau} \Big(\frac{\tau}{2e}\Big)^{-id\tau} \tau^{-1/2} V\Big(\frac{\tau Q}{2\pi N x}\Big) d\tau  +O(N^{-2025}).
\end{equation*}
Expressing the oscillatory factors in exponential form, the resulting expression becomes
\begin{equation} \label{Omega with V_1}
\Omega_x(nN) \asymp (Nn)^{1/2}\int\limits_{\mathbb{R}} e\big(g_1(\tau)\big) V_1(\tau)d\tau+O(N^{-2025}),
\end{equation} 
where the weight function 
\begin{equation*}
V_1(\tau):=\tau^{-1/2} V\Big(\frac{\tau Q}{2\pi N x}\Big)
\end{equation*}
is compactly supported smooth functions on the interval  $[2\pi Nx/Q,\, 4\pi Nx/Q]$, and the phase function is given by
\begin{equation*}
g_1(\tau):=\frac{\tau}{2\pi} \log\Big(\frac{(2\pi)^{(d-1)} n Q}{\tau^{(d-1)} x}\Big)+\frac{(d-1)\tau}{2\pi}. 
\end{equation*}
The first and second derivatives of $g_1$ are
$$
g_1'(\tau)=\frac{1}{2\pi}\log\Big(\frac{(2\pi)^{(d-1)}Qn}{x}\Big)-\frac{d-1}{2\pi}\log \tau
$$
and  
$$
g_1''(\tau)=-\frac{d-1}{2\pi \tau}.
$$
The stationary point $\tau_0$ is determined by $g_1'(\tau_0)=0$. This gives $\tau_0^{(d-1)}=(2\pi)^{(d-1)}\frac{Qn}{x}$, which implies $\tau_0=2\pi \big( \frac{Qn}{x}\big)^{\frac{1}{d-1}}$. By repeated integration by parts, the integral $\Omega(nN)$ is negligibly small unless $\tau_0 \asymp Nx/Q$. Note that the stationary point satisfies $\tau_0 \asymp N/H \asymp Nx/Q$, which in particular lies in the support of $V_1$. Hence, by Lemma \ref{Stationary2} with $U_0\asymp U_1\asymp R\asymp Nx/Q$, $X\asymp (Nx/Q)^{-1/2}$ and $Y\asymp Nx/Q$, the expression in \eqref{Omega with V_1} becomes 
\begin{equation} \label{Omega with V_1 LAST} 
\begin{split}
\Omega_x(nN) &\asymp (Nn)^{1/2} \frac{1}{\sqrt{1/\tau_0}} e\Big( \frac{\tau_0}{2\pi} \log\Big(\frac{(2\pi)^{(d-1)} n Q}{\tau_0^{(d-1)} x}\Big)+\frac{(d-1)\tau_0}{2\pi} \Big) V_1(\tau_0) +O(N^{-2025})\\
&\asymp (Nn)^{1/2} \, e\Big((d-1) \Big(\frac{Qn}{x} \Big)^{\frac{1}{d-1}}\Big)\, V\Big(\big(Q/x\big)^{\frac{d}{d-1}} n^{\frac{1}{d-1}} N^{-1}\Big) +O(N^{-2025}).
\end{split}
\end{equation}
Combining \eqref{Def Xi}, \eqref{cut of n}, and  \eqref{Omega with V_1 LAST}, we arrive at the following lemma.

\begin{lemma}\label{Lemma dual sums}
Under the above notations, we have 
\begin{equation*}
\begin{split}
\sum\limits_{n=1}^{\infty}A_\pi(n) V\left(\frac{n}{N}\right)e\left(\frac{nx}{Q}\right) & \asymp N^{1/2}\sum\limits_{\mathcal{V}^\ast}\sum\limits_{n=1}^{\infty} \frac{A_{\tilde{\pi}}(n)}{n^{1/2}} V_0\Big(\frac{n}{\Tilde{N}}\Big)   \\
& \quad \times e\Big( (d-1) \Big(\frac{Qn}{x} \Big)^{\frac{1}{d-1}}\Big) \,V\Big(\big(Q/x\big)^{\frac{d}{d-1}} n^{\frac{1}{d-1}} N^{-1}\Big) +O(N^{-2025}).
\end{split}
\end{equation*}
\end{lemma}

Applying the above Lemma~\ref{Lemma dual sums} to the two inner sums appearing in $B(d_1,d_2;H,N)$ in equation~\eqref{B sum q=1}, and absorbing the sum over $\mathcal V^\ast$ into the factor $N^\varepsilon$, we obtain 
\begin{equation*}
\begin{split}
B(d_1,d_2;H,N) &\ll \frac{(NM)^{1/2}(NM)^\varepsilon}{Q}\int\limits_{\mathbb{R}} B_0(x)g(1,x)\hat{W}\Big(\frac{xH}{Q}\Big) \\
& \times \sum\limits_{n\sim \Tilde{N}}\frac{A_{\tilde{\pi}_1}(n)}{n^{1/2}} V\Big(\big(Q/x\big)^{\frac{d_1}{d_1-1}} n^{\frac{1}{d_1-1}} N^{-1}\Big) e\Big((d_1-1) \left(\frac{Qn}{x}\right)^{\frac{1}{d_1-1}}\Big) \\
&\times \sum\limits_{m\sim \Tilde{M}}\frac{A_{\tilde{\pi}_2}(m)}{m^{1/2}}U\Big(\big(Q/x\big)^{\frac{d_2}{d_2-1}}\, m^{\frac{1}{d_2-1}}\,M^{-1}\Big)e\Big(-(d_2-1) \left(\frac{Qm}{x}\right)^{\frac{1}{d_2-1}}\Big) dx +O(N^{-2025}),
\end{split} 
\end{equation*} 
where 
\begin{equation}\label{def of tilde N, M}
\frac{(N/H)^{d_1}N^{-\varepsilon}}{N} \ll \Tilde{N}\ll \frac{(N/H)^{d_1}N^\varepsilon}{N} \mbox{ and } \frac{(M/H)^{d_2}M^{-\varepsilon}}{M} \ll  \Tilde{M}\ll \frac{(M/H)^{d_2}M^\varepsilon}{M}.   
\end{equation}
We now set
\begin{equation*}
 G(x):=V\Big(x^{-\frac{d_1}{d_1-1}}\, n^{\frac{1}{d_1-1}}\,H^{-\frac{d_1}{d_1-1}}\,N^{-1}\Big) U\Big(x^{-\frac{d_2}{d_2-1}}\,m^{\frac{1}{d_2-1}}\,H^{-\frac{d_2}{d_2-1}}\,M^{-1}\Big).   
\end{equation*}
Since $n^{\frac{1}{d_1-1}}\,H^{-\frac{d_1}{d_1-1}}\,N^{-1}\asymp1$ and $m^{\frac{1}{d_2-1}}\,H^{-\frac{d_2}{d_2- 1}}\,M^{-1}\asymp1$, it follows that $G(x)$ is a compactly supported smooth function of $x$ whose derivatives are bounded. Consequently, the above becomes
\begin{equation*}
\begin{split}
B(d_1,d_2;H,N) &\ll \frac{(NM)^{1/2}(NM)^\varepsilon}{Q}\int\limits_{\mathbb{R}} B_0(x)g(1,x)\hat{W}\Big(\frac{xH}{Q}\Big) \sum\limits_{n\sim \Tilde{N}}\frac{A_{\tilde{\pi}_1}(n)}{n^{1/2}} e\Big((d_1-1) \left(\frac{Qn}{x}\right)^{\frac{1}{d_1-1}}\Big) \\
&\times \sum\limits_{m\sim \Tilde{M}}\frac{A_{\tilde{\pi}_2}(m)}{m^{1/2}} e\Big(-(d_2-1) \left(\frac{Qm}{x}\right)^{\frac{1}{d_2-1}}\Big)\, G\Big(\frac{xH}{Q}\Big)\, dx +O(N^{-2025}).
\end{split} 
\end{equation*} 
We may replace $g(1,x)$ by $1$ at the cost of a negligible error. Both $B_0(x)$ and $\hat{W}\!\left(\frac{xH}{Q}\right)$ regulate $x$, but the latter imposes the stronger restriction; 
hence we keep $\hat{W}\!\left(\frac{xH}{Q}\right)$ as the weight function. Therefore, we have
\begin{equation}\label{B(d_1,d_2;H,N) last estimate}
\begin{split}
B(d_1,d_2;H,N) &\ll\frac{(NM)^{1/2} (NM)^\varepsilon}{Q}\int\limits_{\mathbb{R}} \hat{W}\Big(\frac{xH}{Q}\Big) \sum\limits_{n\sim \Tilde{N}}\frac{A_{\tilde{\pi}_1}(n)}{n^{1/2}} e\Big((d_1-1) \left(\frac{Qn}{x}\right)^{\frac{1}{d_1-1}}\Big) \\
&\times \sum\limits_{m\sim \Tilde{M}}\frac{A_{\tilde{\pi}_2}(m)}{m^{1/2}} e\Big(-(d_2-1) \left(\frac{Qm}{x}\right)^{\frac{1}{d_2-1}}\Big) \, G\Big(\frac{xH}{Q}\Big)\, dx +O(N^{-2025}).
\end{split} 
\end{equation} 

\section{Proof of Theorem \ref{Main Th 1}}
In this section, we begin with the case in which $\pi_{1}$ and $\pi_{2}$ are
Hecke--Maass cusp forms for $\mathrm{SL}(d,\mathbb{Z})$, with $d=d_{1}=d_{2}\ge 4$, corresponding to the setting of our first main result. Consequently, we can simplify the expression $B(d,d;H,N)$ in \eqref{B(d_1,d_2;H,N) last estimate} as follows.
\begin{equation} \label{Th1: B(d_1,d_2;H,N) with x}
\begin{split}
B(d,d;H,N) &\ll\frac{(NM)^{1/2} (NM)^\varepsilon}{Q}\sum\limits_{n\sim \Tilde{N}}\sum\limits_{m\sim \Tilde{M}}\frac{A_{\tilde{\pi}_1}(n) A_{\tilde{\pi}_2}(m)}{\sqrt{mn}} I(\cdot) +O(N^{-2025}) ,
\end{split}
\end{equation}
where 
\begin{equation*}
I(\cdot)= \int\limits_{\mathbb{R}} \hat{W}\Big(\frac{xH}{Q}\Big) G\Big(\frac{xH}{Q}\Big)\, e\Big((d-1)\left(n^{\frac{1}{d-1}}- m^{\frac{1}{d-1}}\right) \Big(\frac{Q}{x}\Big)^{\frac{1}{d-1}}\Big) dx
\end{equation*}
and 
\begin{equation*}
\frac{(N/H)^{d}N^{-\varepsilon}}{N} \ll \Tilde{N}\ll \frac{(N/H)^{d}N^\varepsilon}{N} \mbox{ and } \frac{(M/H)^{d}M^{-\varepsilon}}{M} \ll  \Tilde{M}\ll \frac{(M/H)^{d}M^\varepsilon}{M}.   
\end{equation*}
Since $N\asymp M$, it follows that 
\begin{equation} \label{thm1: size tilde N , and M}
\tilde{N} \asymp \tilde{M} \asymp \frac{(N/H)^{d}}{N}.
\end{equation}
Making the change of variables $\tfrac{xH}{Q}=y$, we obtain
\begin{equation*}
I(\cdot)=\frac{Q}{H}\int\limits_{\mathbb{R}} \hat{W}(y) G(y)\, e\Big((d-1)\big(n^{\frac{1}{d-1}}- m^{\frac{1}{d-1}}\big)H^{\frac{1}{d-1}}\, y^{-\frac{1}{d-1}}\Big) dy.
\end{equation*}  
By repeated integration by parts, we observe that $I(\cdot)$ is negligibly small unless  $$(d-1)\big(n^{\frac{1}{d-1}}- m^{\frac{1}{d-1}}\big)H^{\frac{1}{d-1}}\ll N^{\varepsilon}.
$$
Equivalently, $I(\cdot)$ is negligibly small unless
\begin{equation} \label{def of L_0}
\begin{split}
m-n &\ll N^\varepsilon H^{-\frac{1}{d-1}}\tilde{N}^{\frac{d-2}{d-1}} \\
&\ll N^\varepsilon \cdot \frac{H}{N} \cdot \tilde{N}:= L_0,\\
\end{split}
\end{equation}
where we used the fact $\tilde{N}\asymp\tilde{M} \asymp \frac{(N/H)^d}{N}$ (see \eqref{thm1: size tilde N , and M}). 
Thus \eqref{Th1: B(d_1,d_2;H,N) with x} implies, up to a negligible error term,
\begin{equation*}
\begin{split}
B(d,d;H,N) &\ll \frac{(NM)^{1/2} (NM)^\varepsilon}{H} \, \sum\limits_{n\sim \Tilde{N}}\sum\limits_{\substack{m\sim \Tilde{M}\\ n-m\ll L_0}}\frac{A_{\tilde{\pi}_1}(n) A_{\tilde{\pi}_2}(m)}{\sqrt{mn}} +O(N^{-2025})\\
&\ll \frac{(NM)^{1/2} (NM)^\varepsilon}{H (\Tilde{N}\Tilde{M})^{1/2}} \, \sum\limits_{n\sim \Tilde{N}}\sum\limits_{\substack{m\sim \Tilde{M}\\ n-m\ll L_0}} A_{\tilde{\pi}_1}(n) A_{\tilde{\pi}_2}(m) +O(N^{-2025}). \\
\end{split}
\end{equation*}
Applying the Cauchy–Schwarz inequality, we obtain
\begin{equation*}
\begin{split}
B(d,d;H,N) &\ll  \frac{(NM)^{1/2} (NM)^\varepsilon}{H (\Tilde{N}\Tilde{M})^{1/2}} \, \Big(\sum\limits_{n\sim \Tilde{N}} |A_{\tilde{\pi}_1}(n)|^2\Big)^{1/2} \Big( \sum\limits_{n\sim \Tilde{N}} \Big|\sum\limits_{\substack{m\sim \Tilde{M}\\ n-m\ll L_0}} A_{\tilde{\pi}_2}(m) \Big|^2 \Big)^{1/2} \\
&\ll \frac{(NM)^{1/2} (NM)^\varepsilon L_0^{1/2}}{H (\Tilde{N}\Tilde{M})^{1/2}} \, \Big(\sum\limits_{n\sim \Tilde{N}} |A_{\tilde{\pi}_1}(n)|^2\Big)^{1/2} \Big( \sum\limits_{n\sim \Tilde{N}} \sum\limits_{\substack{m\sim \Tilde{M}\\ n-m\ll L_0}} |A_{\tilde{\pi}_2}(m)|^2 \Big)^{1/2}.
\end{split}
\end{equation*}
Using Lemma \ref{lemma Rama Bound} to $n$-sum and $m$-sum, we get 
\begin{equation*}
\begin{split}
B(d,d;H,N) &\ll \frac{(NM)^{1/2} (NM)^\varepsilon L_0^{1/2}}{H (\Tilde{N}\Tilde{M})^{1/2}} \, (L_0^{1/2+\varepsilon/2} \tilde{N}^{1+\varepsilon/2}).
\end{split}
\end{equation*}
Using $N\asymp M$ and $\tilde{N}\asymp\tilde{M}$, we arrive at
\begin{equation*}
B(d,d;H,N) \ll  \frac{NL_0^{1+\varepsilon} N^\varepsilon}{H}.
\end{equation*}
Recalling $L_0$ from \eqref{def of L_0}, and using $\tilde{N}\ll \frac{(N/H)^dN^{\varepsilon}}{N} $, we obtain
$$
B(d,d;H,N) \ll N^\varepsilon\cdot\frac{N}{H}\cdot \frac{(N/H)^d}{N} \cdot \frac{H}{N} = N^{\varepsilon} (N^{d-1}{H^{-d}}).
$$ 
This completes the proof of Theorem~\ref{Main Th 1}.

\section{Proof of Theorem \ref{Main Th 2}}
In this section, we consider the case in which $\pi_{1}$ and $\pi_{2}$ are Hecke--Maass cusp forms for $\mathrm{SL}(d_{1},\mathbb{Z})$, $\mathrm{SL}(d_{2},\mathbb{Z})$, respectively, with $d_1>d_2$ and $d_{1},d_{2}\ge 4$, 
corresponding to the setting of our second main result. Consequently, we can simplify the expression $B(d_1,d_2;H,N)$ in \eqref{B(d_1,d_2;H,N) last estimate} as follows.
\begin{equation} \label{2. B(d_1,d_2;H,N) voronoi}
B(d_1,d_2;H,N)\ll \frac{(NM)^{1/2+\varepsilon}}{Q}\sum\limits_{n\sim \Tilde{N}}\frac{A_{\tilde{\pi}_1}(n)}{\sqrt{n}} \sum\limits_{m\sim \Tilde{M}}\frac{A_{\tilde{\pi}_2}(m)}{\sqrt{m}} \mathcal{J}(\cdot)+O(N^{-2025}),
\end{equation}
where 
\begin{equation} \label{th2: J integral}
\mathcal{J}(\cdot):= \int\limits_{\mathbb{R}} \hat{W}\Big(\frac{xH}{Q}\Big) G\Big(\frac{xH}{Q}\Big)\, e\Big((d_1-1)\Big(\frac{Qn}{x}\Big)^{\frac{1}{d_1-1}}-(d_2-1) \Big(\frac{Qm}{x}\Big)^{\frac{1}{d_2-1}}\Big) dx
\end{equation}
and 
\begin{equation}\label{thm2: range N,M tilde}
\frac{(N/H)^{d_1}N^{-\varepsilon}}{N} \ll \Tilde{N}\ll \frac{(N/H)^{d_1}N^\varepsilon}{N} \mbox{ and } \frac{(M/H)^{d_2}M^{-\varepsilon}}{M} \ll  \Tilde{M}\ll \frac{(M/H)^{d_2}M^\varepsilon}{M}.   
\end{equation}
Making a change of variable $y=\frac{xH}{Q}$ in \eqref{th2: J integral}, we obtain
\begin{equation} \label{thm: J int}
\mathcal{J}(\cdot)= \frac{Q}{H}\int\limits_{\mathbb{R}} \hat{W}(y) G(y)\, e\Big((d_1-1)\Big(\frac{Hn}{y}\Big)^{\frac{1}{d_1-1}}-(d_2-1) \Big(\frac{Hm}{y}\Big)^{\frac{1}{d_2-1}}\Big) dy.  
\end{equation}
Here the phase function is given by
$$
f(y)=(d_1-1)\Big(\frac{Hn}{y}\Big)^{\frac{1}{d_1-1}}-(d_2-1) \Big(\frac{Hm}{y}\Big)^{\frac{1}{d_2-1}}.
$$
The derivatives of $f(y)$ are given by
$$
f'(y)=-\frac{1}{y}\Big(\frac{Hn}{y}\Big)^{\frac{1}{d_1-1}}+ \frac{1}{y}\Big(\frac{Hm}{y}\Big)^{\frac{1}{d_2-1}}
$$
and
$$
f''(y)=\frac{d_1}{d_1-1}(nH)^{\frac{1}{d_1-1}}y^{-\frac{2d_1-1}{d_1-1}}-\frac{d_2}{d_2-1}(mH)^{\frac{1}{d_2-1}}y^{-\frac{2d_2-1}{d_2-1}}.
$$
The stationary point of the phase function $f(y)$ is given by 
$$
y_0=H\Big(\frac{m^{d_1-1}}{n^{d_2-1}}\Big)^{\frac{1}{d_1-d_2}}.
$$
By repeated integration by parts, the integral in \eqref{thm: J int} is negligibly small unless $y_0 \asymp 1$. Putting $n\sim \tilde{N}$ and $m\sim \tilde{M}$, and recalling \eqref{thm2: range N,M tilde}, we observe that stationary point satisfies 
$$
y_0\asymp H \Big(\frac{M^{d_2-1}}{H^{d_2}}\Big)^{\frac{d_1-1}{d_1-d_2}} \Big(\frac{N^{d_1-1}}{H^{d_1}}\Big)^{-\frac{d_2-1}{d_1-d_2}} \asymp H \Big(\frac{N^{d_2-1}}{H^{d_2}}\Big)^{\frac{d_1-1}{d_1-d_2}} \Big(\frac{N^{d_1-1}}{H^{d_1}}\Big)^{-\frac{d_2-1}{d_1-d_2}} \asymp 1,
$$
which in particular lies in the support of $\hat{W}G$. Hence, by Lemma~\ref{Stationary2} with $X \asymp U_0 \asymp  U_1 \asymp R \asymp 1, \,\, Y \asymp N/H$, the expression in \eqref{th2: J integral} becomes
\begin{equation} \label{J(.) 0} 
\mathcal{J}(\cdot)\asymp \frac{Q}{H}\cdot\frac{e\left((d_1-d_2)\left(\frac{n}{m}\right)^{\frac{1}{d_1-d_2}}\right) W_1(y_0)}{\sqrt{f''(y_0)}}+ O(N^{-2026}),
\end{equation} 
where $W_1$ is a new compactly supported smooth function with bounded derivatives.
We further observe that 
\begin{equation*}
\begin{split}
f''(y_0)&=\frac{d_1}{d_1-1}(nH)^{\frac{1}{d_1-1}} \Big(H\Big(\frac{m^{d_1-1}}{n^{d_2-1}}\Big)^{\frac{1}{d_1-d_2}}\Big)^{-\frac{2d_1-1}{d_1-1}}-\frac{d_2}{d_2-1}(mH)^{\frac{1}{d_2-1}}\Big(H\Big(\frac{m^{d_1-1}}{n^{d_2-1}}\Big)^{\frac{1}{d_1-d_2}}\Big)^{-\frac{2d_2-1}{d_2-1}}\\
&=\frac{d_2-d_1}{(d_1-1)(d_2-1)} \, H^{-2} \, n^{\theta_1} \,m^{\theta_2},
\end{split}
\end{equation*}
where
\begin{equation} \label{theta_1,2}
\theta_1:= \frac{2d_2-1}{d_1-d _2} \quad \mbox{ and } \quad \theta_2:=-\frac{2d_1-1}{d_1-d_2}.
\end{equation}
Substituting this into \eqref{J(.) 0}, the expression becomes 
\begin{equation} \label{last J(.)}
\mathcal{J}(\cdot)\asymp Q \cdot\frac{e\Big((d_1-d_2)\big(\frac{n}{m}\big)^{\frac{1}{d_1-d_2}}\Big) W_1(y_0)}{n^{\theta_1/2}m^{\theta_2/2}}+ O(N^{-2026}).
\end{equation}
Combining \eqref{2. B(d_1,d_2;H,N) voronoi} and \eqref{last J(.)}, we obtain
\begin{equation}\label{2.B sum last}
B(d_1,d_2;H,N)\ll (NM)^{1/2+\varepsilon} \cdot \sum\limits_{n\sim \Tilde{N}}\frac{A_{\tilde{\pi}_1}(n)}{n^{\frac{\theta_1+1}{2}}} \sum\limits_{m\sim \Tilde{M}}\frac{A_{\tilde{\pi}_2}(m)}{m^{\frac{\theta_2+1}{2}}} e\Big((d_1-d_2)\left(\frac{n}{m}\right)^{\frac{1}{d_1-d_2}}\Big) +O(N^{-2025}).
\end{equation}

\begin{remark}
At this stage, after applying the Cauchy–Schwarz inequality and estimating trivially, followed by an application of Lemma \ref{lemma Rama Bound},  we get $B(d_1,d_2;H,N)\ll N^\varepsilon (N^{(d_1+d_2-1)/2} H^{-(d_1+d_2+1)/2} )$. Hence, we get a power saving bound of the sum $B(d_1,d_2;H,N)$, provided that $H\ge N^{1-\frac{4}{d_1+d_2+1}+\varepsilon}$. This lower bound for the shift $H$ obtained here does not go beyond the bound derived from the result of Friedlander and Iwaniec (see \eqref{Friedlander Iwaniec result}).
\end{remark}
To go beyond this barrier, we impose an additional assumption (without loss of generality) that the shift parameter $H$ lies below the threshold $N^{(d_{2}-1)/(d_{2}+1)}$. Under this assumption and the condition that $d_2\ge4$, it is easy to check that 
\begin{equation}\label{Th2: condition: H/N}
H< N^{\frac{d_2-2}{d_2-1}}<N^{\frac{d_1-2}{d_1-1}}.  
\end{equation}

Applying the Cauchy-Schwarz inequality to \eqref{2.B sum last}, we get
\begin{equation} \label{eq with omega}
 B(d_1,d_2;H,N)\ll (NM)^{1/2+\varepsilon} \cdot \frac{1}{\Tilde{N}^{\frac{\theta_1+1}{2}}} \Big(\sum\limits_{n\sim \Tilde{N}} |A_{\tilde{\pi}_1}(n)|^2\Big)^{1/2}\sqrt{\Upsilon} +O(N^{-2025}), 
\end{equation} 
where 
$$
\Upsilon:=\sum\limits_{n\sim \Tilde{N}} \Big| \sum\limits_{m\sim \Tilde{M}}\frac{A_{\tilde{\pi}_2}(m)}{m^{\frac{\theta_2+1}{2}}} e\Big((d_1-d_2)\left(\frac{n}{m}\right)^{\frac{1}{d_1-d_2}}\Big)\Big|^2.
$$
We now apply the Poisson summation formula to the $n$-sum. To do this, we plug in an approximate smooth bump function, say, $W_3$. Opening the absolute value square, we obtain 
\begin{equation}\label{omega sum}
\Upsilon= \sum\limits_{m_1\sim \Tilde{M}} \sum\limits_{m_2\sim \Tilde{M}}\frac{A_{\tilde{\pi}_2}(m_1)\overline{A_{\tilde{\pi}_2}(m_2)}}{(m_1m_2)^{\frac{\theta_2+1}{2}}}\Sigma,
\end{equation}
where 
$$
\Sigma= \sum\limits_{n\in\mathbb{Z}}  W_3\Big(\frac{n}{\Tilde{N}}\Big)e\Big((d_1-d_2)\Big(\frac{n}{m_1}\Big)^{\frac{1}{d_1-d_2}}- (d_1-d_2)\Big(\frac{n}{m_2}\Big)^{\frac{1}{d_1-d_2}}\Big).$$
Here the size of the phase function is $\asymp \frac{N}{H}$.
Apply the Poisson summation formula on the $n$-sum, we get
\begin{equation}\label{after poisson Sigma}
\Sigma= \sum\limits_{n\in \mathbb{Z}} \int\limits_{\mathbb{R}}W_3\Big(\frac{t}{\Tilde{N}}\Big) e\Big((d_1-d_2)t^{\frac{1}{d_1-d_2}}\big(m_1^{-\frac{1}{d_1-d_2}}- m_2^{-\frac{1}{d_1-d_2}}\big)-nt\Big)dt.    
\end{equation}
Changing the variable $t=t_1 \Tilde{N}$, we have
\begin{equation} \label{th 2: after poisson on n sum}
\Sigma= \Tilde{N}\sum\limits_{n\in \mathbb{Z}} \int\limits_{\mathbb{R}}W_3(t_1) e\Big((d_1-d_2) (t_1\Tilde{N})^{\frac{1}{d_1-d_2}}\big(m_1^{-\frac{1}{d_1-d_2}}- m_2^{-\frac{1}{d_1-d_2}}\big)-nt_1 \Tilde{N}\Big)dt_1.  
\end{equation}
The condition \eqref{Th2: condition: H/N} forces $\Tilde{N}\gg\frac{N^{d_1-1}N^{-\varepsilon}}{H^{d_1}}>N/H$. By repeated integration by parts, we see that only the zero-frequency (i.e., $n=0$) contributes significantly.  
\subsection{The zero frequency} 
Putting $n=0$ in \eqref{th 2: after poisson on n sum}, we have 
\begin{equation*}
\Sigma=\Tilde{N}\int\limits_{\mathbb{R}}W(t_1) e\Big((d_1-d_2)(t_1 \Tilde{N})^{\frac{1}{d_1-d_2}}\big(m_1^{-\frac{1}{d_1-d_2}}- m_2^{-\frac{1}{d_1-d_2}}\big)\Big)dt_1.  
\end{equation*} 
By repeated integration by parts, we observe that $\Sigma$ is negligibly small unless 
$$
\Tilde{N}^{\frac{1}{d_1-d_2}}\big(m_1^{-\frac{1}{d_1-d_2}}- m_2^{-\frac{1}{d_1-d_2}}\big)\ll N^\varepsilon.
$$
Equivalently, $\Sigma$ is negligibly small unless
\begin{equation}\label{K_0 defini}
\begin{split} 
m_1-m_2 &\ll N^{\varepsilon} \cdot \Tilde{N}^{-\frac{1}{d_1-d_2}} \Tilde{M}^{1+\frac{1}{d_1-d_2}} \\
& \ll N^{\varepsilon} \cdot \Big(\frac{N^{d_1-1}}{H^{d_1}}\Big)^{-\frac{1}{d_1-d_2}} \Big(\frac{N^{d_2-1}}{H^{d_2}}\Big)^{\frac{1}{d_1-d_2}} \Tilde{M}\\
&  \ll N^{\varepsilon} \cdot \frac{H}{N} \cdot \Tilde{M}:=K_0,
\end{split}
\end{equation}
where we used \eqref{thm2: range N,M tilde}.
Thus \eqref{omega sum} implies, up to a negligible error term
\begin{equation*} 
\Upsilon= \Tilde{N} \sum\limits_{m_1\sim \Tilde{M}} \sum\limits_{\substack{m_2\sim \Tilde{M}\\ m_1-m_2 \ll K_0 }}\frac{A_{\tilde{\pi}_2}(m_1)\overline{A_{\tilde{\pi}_2}(m_2)}}{(m_1m_2)^{\frac{\theta_2+1}{2}}}+O(N^{-2025}). 
\end{equation*} 
Applying the Cauchy–Schwarz inequality, we obtain
\begin{equation*}
\begin{split}
\Upsilon  &\ll \frac{\Tilde{N} }{\Tilde{M}^{(\theta_2+1)}} \Big(\sum\limits_{m_1\sim \Tilde{M}} |A_{\tilde{\pi}_2}(m_1)|^2\Big)^{1/2}\Big(\sum\limits_{m_1\sim \Tilde{M}} \big| \sum\limits_{\substack{m_2\sim \Tilde{M}\\ m_1-m_2 \ll K_0 }}A_{\tilde{\pi}_2}(m_2)\big|^2\Big)^{1/2} \\
&\ll  \frac{\Tilde{N} K_0^{1/2}}{\Tilde{M}^{(\theta_2+1)}} \Big(\sum\limits_{m_1\sim \Tilde{M}} |A_{\tilde{\pi}_2}(m_1)|^2\Big)^{1/2}\Big( \sum\limits_{m_1\sim \Tilde{M}} \sum\limits_{\substack{m_2\sim \Tilde{M}\\ m_1-m_2 \ll K_0 }}|A_{\tilde{\pi}_2}(m_2)|^2\Big)^{1/2}.
\end{split}
\end{equation*}   
Applying Lemma \ref{lemma Rama Bound} to both the $m_1$- and $m_2$-sums, we get 
\begin{equation} \label{upsilon bdd last}
\begin{split}
\Upsilon &\ll \frac{\Tilde{N} K_0^{1/2}}{\Tilde{M}^{(\theta_2+1)}} (\tilde{M}^{1+\varepsilon/2}K_0^{1/2+\varepsilon/2})\\
& = \tilde{N} K_0^{1+\varepsilon/2} \tilde{M}^{-\theta_2+\varepsilon/2}.
\end{split}
\end{equation} 
Putting \eqref{upsilon bdd last} into \eqref{eq with omega}, and applying Lemma~\ref{lemma Rama Bound} to the $n$-sum in \eqref{eq with omega}, we obtain  
\begin{equation*}
\begin{split}
B(d_1,d_2;H,N) &\ll N^{\varepsilon} (NM)^{1/2} \cdot \frac{\tilde{N}^{1/2}}{\tilde{N}^{(\theta_1+1)/2}} \cdot (\tilde{N} K_0^{1+\varepsilon/2} \tilde{M}^{-\theta_2+\varepsilon/2})^{1/2}   \\
& =N^\varepsilon (NM)^{1/2} \cdot \tilde{N}^{-\theta_1/2}\tilde{M}^{-\theta_2/2} \tilde{N}^{1/2}K_0^{1/2+\varepsilon/4}.
\end{split}
\end{equation*} 
Recalling $K_0$ from \eqref{K_0 defini}, we obtain
\begin{equation} \label{last B}
B(d_1,d_2;H,N)\ll  N^\varepsilon (NM)^{1/2} \cdot \tilde{N}^{-\theta_1/2}\tilde{M}^{-\theta_2/2} \tilde{N}^{1/2} \tilde{M}^{1/2} (H/N)^{1/2}.
\end{equation}
Substituting the values of $\theta_1$ and $\theta_2$ from \eqref{theta_1,2}, recalling \eqref{thm2: range N,M tilde}, and using $N \asymp M$, we obtain 
\begin{equation} \label{bdd1}
\tilde{N}^{-\theta_1/2}\tilde{M}^{-\theta_2/2} \asymp  (NH)^{-1/2}
\end{equation}
and
\begin{equation} \label{bdd2}
\tilde{N}^{1/2} \tilde{M}^{1/2}\asymp N^{(d_1+d_2-2)/2}H^{-(d_1+d_2)/2}.
\end{equation}
Combining \eqref{last B}, \eqref{bdd1}, and \eqref{bdd2}, we deduce
\begin{equation*}
B(d_1,d_2;H,N)\ll  N^\varepsilon N^{\frac{(d_1+d_2-2)}{2}}H^{-\frac{(d_1+d_2)}{2}}.
\end{equation*}
This completes the proof of Theorem~\ref{Main Th 2}. \\

{\bf Acknowledgements.} The author is grateful to Professor Ritabrata Munshi for suggesting the problem and for helpful discussions. The author also gratefully acknowledges the Indian Statistical Institute, Kolkata, for its excellent research environment and support through the Research Associate Fellowship. \\


\begin{thebibliography}{20} 

\bibitem{BBMZ12}
S. Baier, T. D. Browning, G. Marasingha, and L. Zhao; {\it Averages of shifted convolutions of $d_3(n)$,} Proc. Edinb. Math. Soc. (2) {\bf 55} (2012), no.~3, 551--576; MR2975242

\bibitem{Blom04}
V. Blomer, {\it Shifted convolution sums and subconvexity bounds for automorphic $L$-functions,} Int. Math. Res. Not. {\bf 2004}, no.~73, 3905--3926; MR2104288

\bibitem{BlKhYo2013}
V. Blomer, R. Khan and M. P. Young; {\it Distribution of mass of holomorphic cusp forms}, Duke Math. J. {\bf 162} (2013), no. 14, 2609–2644; MR3127809

\bibitem{BlHa08}
V. Blomer and G. Harcos,{\it The spectral decomposition of shifted convolution sums}, Duke Math. J. {\bf 144} (2008), no.~2, 321--339; MR2437682


\bibitem{DLY24} A. Dasgupta, W. H. Leung, and M. P. Young; {\it 
The second moment of the $\mathrm{GL}(3)$ standard $L$-function on the critical line}, arXiv:2407.06962, (2024), [math.NT]

\bibitem{DFI93} W. Duke, J.~B. Friedlander and H. Iwaniec; {\it  Bounds for automorphic $L$-functions,} Invent. Math. {\bf 112} (1993), no.~1, 1--8; MR1207474

\bibitem{DFI94}
W.~D. Duke, J.~B. {\it Friedlander and H. Iwaniec, A quadratic divisor problem}, Invent. Math. {\bf 115} (1994), no.~2, 209--217; MR1258903

\bibitem{FI05}
J.~B. Friedlander and H. Iwaniec, {\it Summation formulae for coefficients of $L$-functions}, Canad. J. Math. {\bf 57} (2005), no.~3, 494--505; MR2134400

\bibitem{Gold06b}
D.~M. Goldfeld; {\it Automorphic forms and $L$-functions for the group ${\rm \mathrm{GL}}(n,\bold R)$}, Cambridge Studies in Advanced Mathematics, 99, Cambridge Univ. Press, Cambridge, 2006; MR2254662

\bibitem{HardyLittle1916}
G. H. Hardy and J. E. Littlewood; {\it Contributions to the theory of the Riemann zeta-function and the theory of the distribution of primes}, Acta Math. 41 (1916), no.1, 119-196; MR1555148

\bibitem{Harun24a}
M. Harun and S.~K. Singh; A shifted convolution sum for $\mathrm{GL}(3)\times \mathrm{GL}(2)$ with weighted average, Ramanujan J. {\bf 64} (2024), no.~1, 93--122; MR4740261

\bibitem{Harun24b}
M. Harun and S.~K. Singh; {\it Shifted convolution sum with weighted average: $\mathrm{GL}(3)\times \mathrm{GL}(3)$ setup,} J. Number Theory {\bf 261} (2024), 55--94; MR4724035

\bibitem{Holo09}
R. Holowinsky; {\it A sieve method for shifted convolution sums}, Duke Math. J. {\bf 146} (2009), no.~3, 401--448; MR2484279

\bibitem{Huang2021}
B. R. Huang; {\it On the Rankin-Selberg problem}, Math. Ann. {\bf 381} (2021), no. 3-4, 1217–1251; MR4333413

\bibitem{IwaKowBook}
H. Iwaniec and E. Kowalski; {\it Analytic number theory,} American Mathematical Society Colloquium Publications, {\bf 53}, Amer. Math. Soc., Providence, RI, 2004; MR 2061214

\bibitem{Michel07} P. Michel; Analytic number theory and families of automorphic $L$-functions, in {\it Automorphic forms and applications}, IAS/Park City Math. Ser., 12, Amer. Math. Soc., Providence, RI, (2007) pp. 181–295; MR2331346

\bibitem{Michel04}
P. Michel,{\it The subconvexity problem for Rankin-Selberg $L$-functions and equidistribution of Heegner points}, Ann. of Math. (2) {\bf 160} (2004), no.~1, 185--236; MR2119720

\bibitem{Munshi13}
R. Munshi; {\it Shifted convolution sums for $\mathrm{GL}(3)\times \mathrm{GL}(2)$,} Duke Math. J. {\bf 162} (2013), no.~13, 2345--2362; MR3127803

\bibitem{Munshi22} R. Munshi; {\it  On a shifted convolution sum problem}, J. Number Theory {\bf 230} (2022), 225--232; MR4327955

\bibitem{Paul2021}
P. D. Nelson; {\it Bounds for standard $ L $-functions}, arXiv preprint arXiv:2109.15230, (2021), [math.NT]

\bibitem{PalPal} R. Pal and S. Pal; {\it On shifted convolution sums of $\mathrm{\mathrm{GL}}(3)$-Fourier coefficients with an average over shifts}, arXiv:2510.15799 (2025),[math.NT]

\bibitem{Pitt95} N.~J.~E. Pitt; {\it On shifted convolutions of $\zeta^3(s)$ with automorphic $L$-functions}, Duke Math. J. {\bf 77} (1995), no.~2, 383--406; MR1321063

\bibitem{Sarnak01}
P. Sarnak; {\it Estimates for Rankin-Selberg $L$-functions and quantum unique ergodicity}, J. Funct. Anal. {\bf 184} (2001), no.~2, 419--453; MR1851004

\bibitem{Selberg65}
A. Selberg, {\it On the estimation of Fourier coefficients of modular forms}, in {\it Proc. Sympos. Pure Math., Vol. VIII}, pp. 1--15, Amer. Math. Soc., Providence, RI, ; MR0182610

\bibitem{Xi18}
P. Xi;{ \it  A shifted convolution sum for $\rm \mathrm{GL}(3)\times\rm \mathrm{GL}(2)$,} Forum Math. {\bf 30} (2018), no.~4, 1013--1027; MR3824803

\end{thebibliography}
\end{document}